\documentclass[final]{siamltex}

\usepackage{graphicx}
\usepackage{amsmath,amssymb}
\usepackage{booktabs}
\usepackage{color}
\usepackage{hyperref}
\usepackage{subfig}
\usepackage{cleveref}
\usepackage{txfonts}

\newcommand{\Ra}	{\Rightarrow}

\newcommand{\LRa}{\Leftrightarrow}
\newcommand{\eps}{\varepsilon}

\newcommand {\calC}	{{\cal C}}

\newcommand {\calI}	{{\cal I}}

\newcommand {\calM}	{{\cal M}}

\newcommand {\R}	{\mathbb{R}}

\newcommand {\N}	{\mathbb{N}}

\newcommand{\ds}{\displaystyle}
\newcommand{\g}{\text{\textbf{g}}}

\newtheorem{remark}{Remark} % <- Preamble

\title{Towards differential geometric characterization of slow invariant manifolds in extended state space: Sectional Curvature and Flow Invariance}

\author{Pascal Heiter\footnotemark[2]%\thanks{Scientific Computing Centre Ulm (UZWR), 
        \and Dirk Lebiedz\footnotemark[2]}%\thanks{Scientific Computing Centre Ulm (UZWR), 
\begin{document}

\maketitle

\renewcommand{\thefootnote}{\fnsymbol{footnote}} \footnotetext[2]{Institute for Numerical Mathematics, 
	Ulm University, Helmholtzstra\ss e 20, 89081 Ulm, Germany
  (\texttt{pascal.heiter@uni-ulm.de, dirk.lebiedz@uni-ulm.de}).}
  \renewcommand{\thefootnote}{\arabic{footnote}}

% REQUIRED
\begin{abstract}
  Some model reduction techniques for multiple time-scale dynamical systems make use of the identification of low dimensional slow invariant attracting manifolds (SIAM) in order to reduce the dimensionality of the { phase} space by restriction to the slow flow. The focus of this work is on a proposition and discussion of a general viewpoint using differential geometric concepts for submanifolds to deal with slow invariant manifolds in an extended { phase} space. The motivation is a coordinate independent formulation of the manifold properties and its characterization problem treating the manifold as intrinsic geometric object. We formulate a computationally verifiable necessary condition for the slow invariant manifold graph stated in terms of a differential geometric view on the invariance property. Its application to example systems is illustrated. In addition, we present some ideas and investigations concerning the search for sufficient, differential geometric conditions characterizing slow invariant manifolds based on our previously developed variational principle.
  \end{abstract}

% REQUIRED

% REQUIRED

\pagestyle{myheadings}
\thispagestyle{plain}
\markboth{P.\ Heiter and D.\ Lebiedz}{Towards diff.\ geom.\ characterization SIM}

%%%%%%%%%%%%%%%%%%%%%%%%%%%%%%%%%%%%%%%%%%%%%%
% Section: Introduction
%%%%%%%%%%%%%%%%%%%%%%%%%%%%%%%%%%%%%%%%%%%%%%
\section{Introduction}

Multiple time scale dynamical systems occur widely in modeling natural processes such as chemical reactions often resulting in high-dimensional kinetic ordinary differential equation (ODE) systems. 
The time scales frequently range from nanoseconds to seconds and the numerical simulation of such stiff dynamical systems becomes very time consuming. 
This calls for appropriate model reduction techniques. 
Time scale separation of a dynamical system flow into fast and slow modes is the basis for many model and complexity reduction approaches. 
The fast relaxing modes are approximated by enslaving them to the slow modes via a mapping. 
The mathematical object related to this idea is the invariant attracting manifold of slow motion.

The first model reduction techniques have been the quasi-steady-state assumption (QSSA) \cite{Bodenstein1913, Chapman1913} and the partial equilibrium approximation (PEA) \cite{Michaelis1913}. 
In the QSSA approach, specific variables are supposed to be in steady state, by contrast in the PEA approach it is assumed that certain (fast) reactions are equilibriated. 
Both methods are still used nowadays due to their conceptual simplicity although more sophisticated methods have been developed. 
A very popular and widely used reduction technique is the intrinsic low-dimensional manifold (ILDM) method developed by Maas and Pope in 1992, cf.\ \cite{Maas1992}. 
All aforementioned approaches have in common, that the resulting manifold is not invariant. 
The computational singular perturbation (CSP) \cite{Lam1985, Lam1994} method, which is proposed by Lam in 1985, and equation free methods from Kevrekidis et al.\ \cite{Kevrekidis2003} as well as from Theodoropoulos et al.\ \cite{Theodoropoulos2000}, the relaxation redistribution method \cite{Chiavazzo2011} by Chiavazzo and Karlin and a finite-time Lyapunov exponents based method by Mease et al.\ \cite{Mease2016} are other popular model reduction techniques. 
Some prominent methods for the numerical computation of slow invariant manifolds in context of chemical kinetics are a method for generation of invariant grids \cite{Chiavazzo2007, Gorban2005} and the G-scheme framework by Valorani and Paolucci \cite{Valorani2009}. 
Further approaches are the invariant constrained equilibrium edge preimage curve (ICE-PIC) method introduced by Ren et al.\ \cite{Ren2005, Ren2006a}, the zero-derivative principle (ZDP) method presented by Gear, Zagaris et al.\ \cite{Gear2005, Zagaris2009}, the functional equation truncation (FET) approach by Roussel \cite{Roussel2006, Roussel2012} and methods by Adrover et al.\ \cite{Adrover2007, Adrover2007a} and Al-Khateeb et al.\ \cite{Al-Khateeb2009}. 
The flow curvature method (FCM) {\cite{Ginoux2006, Ginoux2008, Ginoux2009, Ginoux2014} presented by Ginoux}, which is also discussed in the work of Br\o{}ns et al.\ \cite{Brons2013}, is based on a {differential geometric} analysis of curves within the high-dimensional { phase} space. 
The main idea of approximating the codimension-1 slow invariant manifold is the annulation of generalized curvatures of curves based on Frenet frames. 
This interesting approach is based on curves and requires the computation of a determinant of the time derivatives of the state vector, invariance of the compute manifold follows then from the Darboux theorem. 
The work of Ginoux et al.\ is of particular significance in the context of this paper, it supported our inspiration to work on topics presented here although we have been concerned with differential geometry ideas in the context of slow invariant manifolds since some years \cite{Lebiedz2010} without knowing about Ginoux's work.
Our model reduction technique is based on a variational principle by using a trajectory-based optimization approach is proposed by Lebiedz et al.\ \cite{Lebiedz2004, Lebiedz2006b,Lebiedz2010,Lebiedz2011,Lebiedz2011a, Lebiedz2013a}, which is supposed to be applied to kinetic models in combustion chemistry \cite{Lebiedz2013, Lebiedz2014}. 
The recently published work of Lebiedz and Unger \cite{Lebiedz2016} discusses and exploits common ideas and brings together concepts of several model reduction approaches.

{
In \cite{Lebiedz2016} we study fundamental principles underlying various SIAM computation ideas, in particular we distinguish between methods that construct pointwise liftings to the manifold und those using the flow of the dynamical systems. The ZDP is a classical lifting method that does not use the flow but a local root finding criterion to approximate the SIAM. Whereas the ICE-PIC method in principle uses only the flow of the dynamical system computing a manifold point by solving (with a shooting approach) a boundary value problem involving an appropriate point on the edge of the manifold and following a solution trajectory up to a manifold point corresponding to the chosen parameterization values. Our boundary value problem introduced in \cite{Lebiedz2016} belongs to both classes in some respect, it uses the flow and a local criterion for lifting. However, the differential geometric approach based on { phase}-space-time manifolds presented here, is a local method, i.e. it belongs to the lifting class such as ZDP. The flow and flow-induced mappings are only used to characterize local geometric properties (in particular the time-sectional curvature) of the slow invariant manifold in time direction.

The multiscale issue related to a transition from microscopic to macroscopic models via sub has also been addressed by Transtrum et al. \cite{Transtrum2014, Transtrum2016}. Of particular interest in the context of our work is the fact that Transtrum also exploits differential geometric aspects and even established a connection to information geometry taking a statistical viewpoint where a Riemann metric is derived from Fisher-information-type. To study this viewpoint in terms of slow invariant attracting manifolds as approximations of many initial value problems might be an interesting issue for future research.

}

The geometric singular perturbation theory (GSPT) originated by Fenichel \cite{Fenichel1972,Fenichel1974,Fenichel1977,Fenichel1979}{, which is also presented in \cite{Jones1995, Kaper1999}, }is tailored for theoretical analysis of time scale separated systems in singularly perturbed form

\[
	\begin{array}{rcl}
		\ds \frac{d}{dt}x &=& \eps f(x,y,\eps) {\in \R^{n_x}} \\[0.25cm]
		\ds \frac{d}{dt}y &=& g(x,y,\eps) {\in \R^{n_y}}
	\end{array}
\]
omitting potential explicit time $t$-dependencies of the right hand side functions here, involving {$n_x,n_y \in \N,$} a parameter $\eps > 0$ and sufficiently smooth functions {$f: \R^{n_x} \times \R^{n_y} \times (0,\infty) \to \R^{n_x},g : \R^{n_x} \times \R^{n_y} \times (0,\infty) \to \R^{n_y}$.} Fenichel's theory provides a set of theorems to analyze and gain deeper understanding of those system in terms of the slow flow behaviour. The critical manifold $S_0$ is defined as the set of points, where $g(x,y,0) = 0$ holds. The slow invariant manifold, whose existence for sufficiently small $\eps$ has been proved by Fenichel, is defined as
\[
 	S_\eps := \{ (x,y) \;:\; y = h_\eps(x) \}
\]
with an asymptotic expansion in form of a graph
\[
	y = h_\eps(x) = h_0(x) + \eps h_1(x) + \eps^2 h_2(x) + ...
\]
whereby the coefficients $h_i(x)$ can be obtained by iteratively and recursively solving the invariance equation \begin{equation}
\label{eq:invariance}
 \eps Dh_\eps(x)f(x,h_\eps(x),\eps) - g(x,h_\eps(x),\eps) = 0.
\end{equation}

The aim of the present work is the proposition and establishment of a novel differential geometric viewpoint, in which the object of a slow invariant manifold \emph{might} be characterized by purely geometric, coordinate independent properties within an extended { phase}-space-time manifold. 
By doing so, we aim at avoiding a description with an asymptotic expansion in $\eps$ and the implications of a non-uniqueness of the slow invariant manifold defined by matched asymptotic expansion according to Fenichel's theory. 
The hope is, that the whole { phase}-space-time manifold spanned by solution trajectories of the ODE contains all required information to identify an invariant (and appropriately attracting) object, which matches with the Fenichel definition up to all orders in $\eps$. 
In particular, \Cref{sec:neccond} deals with a necessary condition based on a reformulation of the invariance equation \cref{eq:invariance} in terms of differential geometric property which has to be met by the slow invariant manifold. For this purpose we formulate the manifold as a graph over some (slow) variables and the time-axis.
In \Cref{sec:examples}, the theoretical result is applied to several well known examples in order to illustrate the necessary condition. 
Higher-dimensional models with analytically known slow invariant manifold are constructed and analyzed in the following. 
The lack of sufficiency is investigated and some considerations on additional criteria are presented in \Cref{sec:sufficient}. 
\Cref{sec:conclusion} contains a summary and an outlook focused on a way towards a complete characterization of SIAM of arbitrary dimension in terms of its differential geometry.

%%%%%%%%%%%%%%%%%%%%%%%%%%%%%%%%%%%%%%%%%%%%%%
% Section: Necessary Condition for Slow Invariant Manifolds
%%%%%%%%%%%%%%%%%%%%%%%%%%%%%%%%%%%%%%%%%%%%%%
\section{Necessary Condition for Slow Invariant Manifolds}
\label{sec:neccond}
We state a necessary condition for slow invariant manifolds in a differential geometry context. Assume $\eps > 0, n,k \in \N$ with $k < n$ and consider a $(k,n-k)$-(slow,fast) system
{\begin{equation}
	\begin{array}{rcl}
		\ds \dot{x}(t) &=& f(x(t),y(t)) \in \R^k \\[0.2cm]
		\ds \eps \dot{y}(t) &=& g(x(t),y(t)) \in \R^{n-k}\\
	\end{array}
	\label{eq:sfsystem}
\end{equation}
with $x := (x_1,...,x_k), y := (y_{1},...,y_{n-k}), f := (f_1,...,f_k)$ and $g := (g_{1},...,g_{n-k}).$} We define a smooth immersion 
{\begin{equation}
	\begin{array}{lcll}
		\psi: &[0,\infty) \times \R^{k} &\to &\R^{n+1} \\
%		 & (t,\zs) &\mapsto & (t,\zs,p^{(1)}[a^{(1)}](t,\zs),...,p^{(n-k)}[a^{(n-k)}](t,\zs))
		 & (t^*,x^*) &\mapsto & (t^*,x^*,p(t^*,x^*; a))	
	\end{array}
	\label{eq:immersion}
\end{equation}}
involving a sufficiently smooth function
{\[
	\begin{array}{lcll}
		p: & [0,\infty) \times \R^k \times (\calC^\infty(\R^{k}))^{n-k} &\to &\R^{n-k} \\
		 & (t^*,x^*; a) &\mapsto & y^* = p(t^*,x^*; a)	
	\end{array}
\]
where $y^* = y(t^*)$ is the solution of the following boundary value problem (BVP)
\begin{equation}
\label{eq:bvp}
	\begin{cases}	\quad \dot{x}(t) = f(x(t),y(t)) \\
				\quad	\eps \dot{y}(t) = g(x(t),y(t))\\
				\quad	x(t^*) = x^* \\
				\quad	y(0) = a(x(0)).
	\end{cases}
\end{equation}
The function $a(\cdot)$ plays the role as of initial value function coupling the data $x(0)$ and $y(0)$, i.e. lifting the parameterizing values of the slow variables $x$ to a manifold point in the full space. \Cref{fig:P_bvp} illustrates the BVP viewpoint and the meaning of the initial value function (red curve) plotted into the $xy$-plane.
\begin{figure}[tbhp]
	\centering
	\includegraphics[scale=0.42]{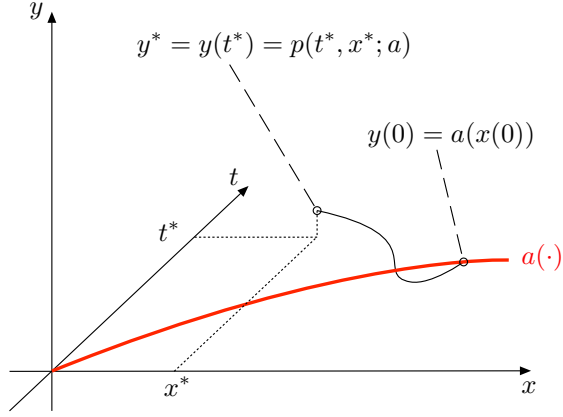}
	\caption{Illustration of the boundary value problem \cref{eq:bvp} with an initial value function $a(\cdot)$ (red curve) plotted into the $xy$-plane and a fixed point $(t^*,x^*)$ as well as {$y^* = p(t^*,x^*;a).$}}
	\label{fig:P_bvp}
\end{figure}
A submanifold is determined by a choice of an arbitrary initial value function $a(u)$ considering all possible initial values $u \in \R^k$ transported by the phase flow of ($\ref{eq:sfsystem}$).} We define the following $k+1$-dimensional manifold
{\[
	\calM  = \calM(a) := \left\{  (t^*,x^*,y^*) \in [0,\infty) \times \R^{n}\;:\; y^* = p(t^*,x^*;a)\right\}
\]}
embedded in { phase}-space-time $[0,\infty) \times \R^{n}$ and being defined as the graph of the time evolution with the flow of $\cref{eq:sfsystem}$ of {all possible initial values $(u,a(u)) \in \R^n, u \in \R^k.$} \Cref{fig:spacetime_ds} illustrates this issue in the case of an (1,1) system in a three-dimensional { phase}-space-time frame with two different initial value functions. 
{\begin{remark} The manifold $\calM$ can also be seen as the union of all trajectories with feasible initial values coupled via $a.$ Especially, it holds that $p(t^*,x^*;a)$ equals the solution of
\[
	\exists\; x_0 \in \mathbb{R}^k \,:\; \Phi^{t^*}(x_0,a(x_0)) = (x^*,y^*)
\] 
where $\Phi^{t^*}(x_0,a(x_0))$ denotes the flow map of the dynamical system \cref{eq:sfsystem} at time $t^*$ starting from the initial value $(x_0,a(x_0)).$
\end{remark}}
\begin{figure}[tbhp]
	\centering
	\subfloat{\includegraphics[scale=0.32]{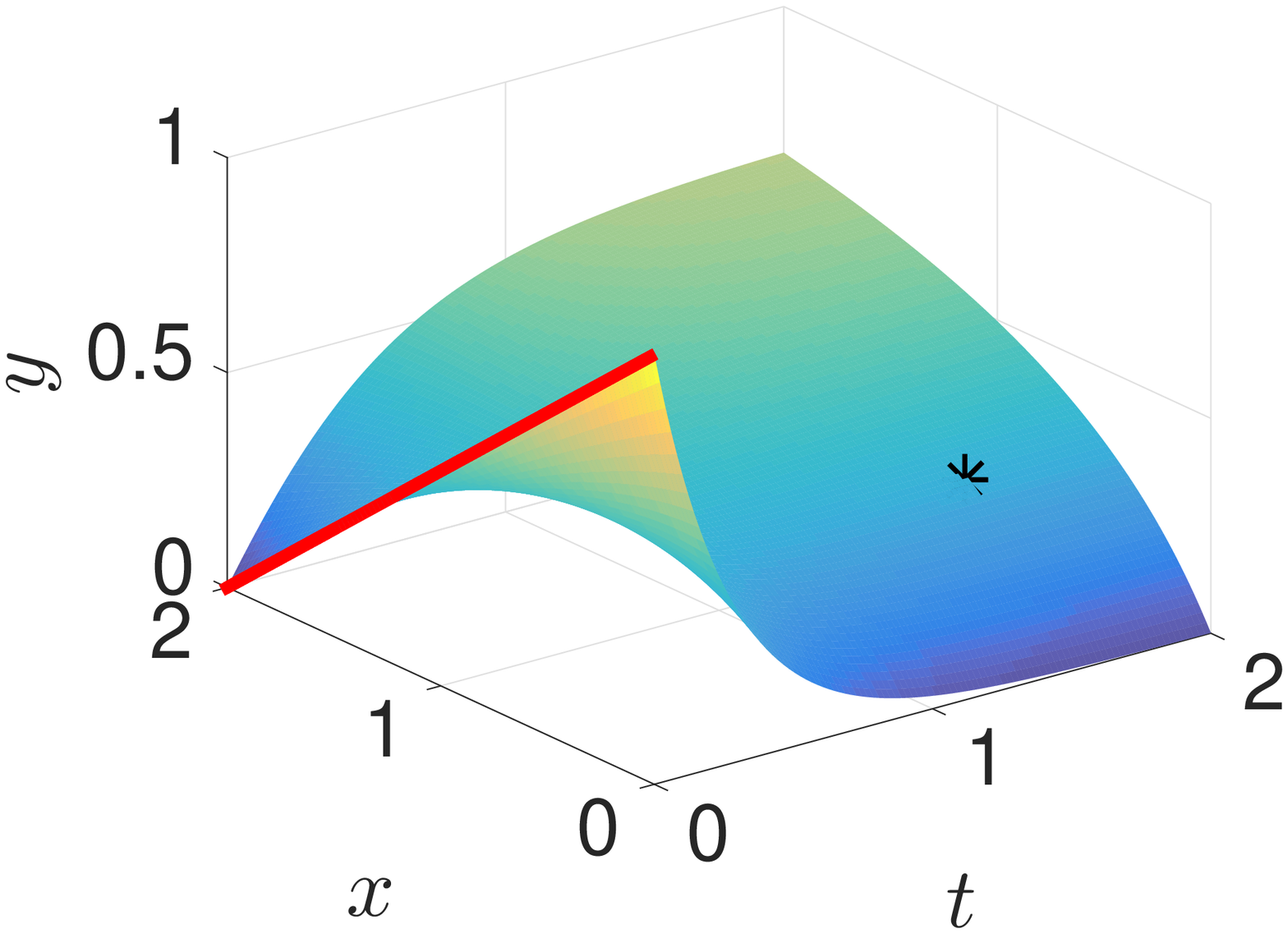}}
	\subfloat{\includegraphics[scale=0.32]{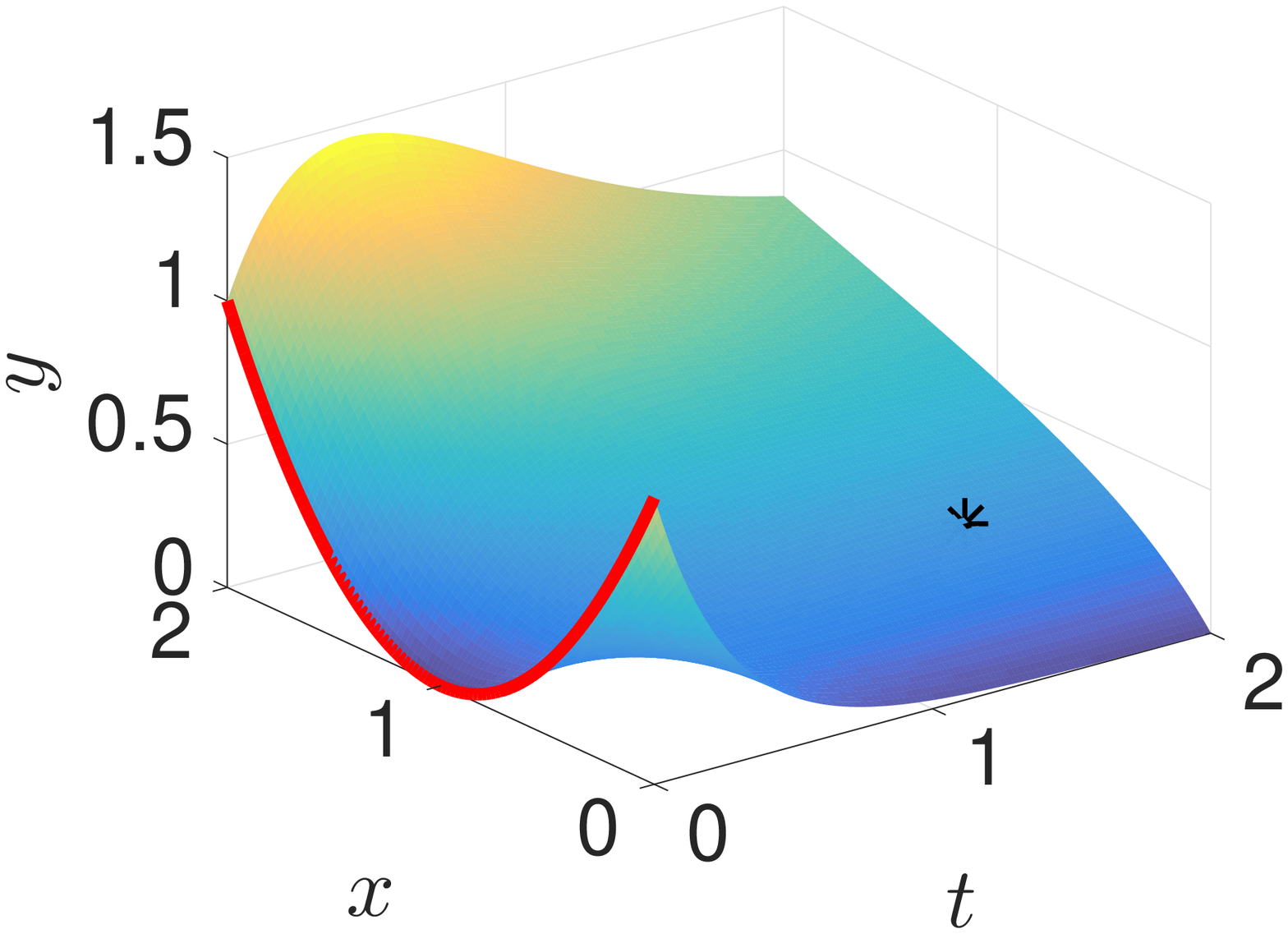}}
	\caption{Illustration of an (1,1) system in { phase}-space-time. Colored surfaces: Subsets of $\calM$ defined for {$(t^*,x^*) \in [0,2]^2$}. Red curves: Different initial value functions $a(u)$ for $u \in [0,2]$ plotted into the $xy$-plane. Black Mark: Selected point {$\psi(1.5,0.5) = (1.5,0.5,p(1.5,0.5;a)) \in \calM$}.}
	\label{fig:spacetime_ds}
\end{figure}
\begin{proposition}[Induced Metric Tensor]
\label{cor:metricTensor}
Consider the smooth immersion \cref{eq:immersion} \\and define
\[
	{\g} := J^TJ \qquad \text{ and } \qquad 
	J := 	\begin{pmatrix}
					{\frac{\partial}{\partial t^*}\psi, \frac{\partial}{\partial x^*_1}\psi,...,\frac{\partial}{\partial x^*_k}\psi}\\
			\end{pmatrix}.
\]
Then {$(\calM,\g)$} is a Riemannian manifold and {$\g$} is a metric tensor induced by the euclidean metric in $\R^{n+1}.$
\end{proposition}
\begin{proof} Since we defined {$\g = J^T J$}, it holds that {$\g$} is the Gramian matrix of {$\{ \frac{\partial}{\partial t^*}\psi, \frac{\partial}{\partial x^*_1}\psi,...,\frac{\partial}{\partial x^*_k}\psi \}$} and symmetric by definition. 
Moreover, {$\g$} is positive semi definite. 
We have to show that {$\g$} is positive definite. It holds
\[
	J = 	\begin{pmatrix}
					I \\
					\star
			\end{pmatrix} \in \R^{n \times k+1}
\]
with the identity matrix $I \in \R^{k+1\times k+1}$ {and arbitrary values $\star$}. Thus $\text{rank}\,{J} = k+1$, which implies {$\det \g \not= 0.$}
Therefore, {$\g$ }is symmetric and positive definite.
\end{proof}
In the following, we denote by
\[
	{T_q}\calM = \text{span}\left\{ \frac{\partial}{\partial x^1}, \frac{\partial}{\partial x^2},...,\frac{\partial}{\partial x^{k+1}} \right\} := {\text{span}\left\{ \frac{\partial}{\partial t^*}\psi, \frac{\partial}{\partial x^*_1}\psi,...,\frac{\partial}{\partial x^*_k}\psi \right\}}
\]
the tangential space of $\calM$ {at point $q \in \calM$}, omit the explicit dependencies on all arguments of {$p(t^*,x^*;a)$} for simplicity and use the Einstein sum notation.
\begin{definition}[Time Sectional Curvatures]
Let {$(\calM,\g)$} be the Riemannian manifold defined in \cref{cor:metricTensor} and {$T_q\calM$} be its tangential space. We define
\[
	\sigma_i := \text{span}\left\{ \frac{\partial}{\partial x^1}, \frac{\partial}{\partial x^{i}} \right\} \subset {T_q\calM} \qquad i=2,...,k+1
\]
with $\dim \sigma_i = 2.$ Then, it holds for the sectional curvatures
\[
	K(\sigma_i) = \frac{R^{m}_{1ii} g_{m1}}{g_{11}g_{ii}-g_{1i}^2}
\]
whereas $R^{m}_{1ii}$ are the entries of the Riemann tensor and $g_{ij}$ are the entries of the metric tensor {$\g$}.
\end{definition}
The set $\{K(\sigma_i),i=2,...,k+1\}$ is the collection of sectional curvatures involving the time direction. The importance of such directions is highlighted in the following proposition.

\begin{proposition}[Necessary Condition for Invariant Manifolds]
\label{thm:neccond}
Consider \\system \cref{eq:sfsystem}. Fenichel's Theorem guarantees the existence of a slow invariant manifold 
\[
	S_\eps = {\left\{ (x,y) \in \R^n \;:\; y = h_\eps(x) \right\}}.
\]
Let $a(\cdot) = h_\eps(\cdot)$. Then, it holds for all ${q} \in \calM$ and $i=2,...,k+1$
\[
	K(\sigma_i) = 0.
\]
\end{proposition}
\begin{proof} {We set $t = t^*$ and $x_i = x^*_i$ for simplicity. Consider 
\[
	h_\eps(x(t)) = y(t) = p(t,x(t);h_\eps).
\]
Differentiating this equation with respect to $t$ leads to
\[
	\begin{array}{rcl}
		\ds \frac{d}{dt} y(t) = \ds\dot{y}(t) &=& \ds \frac{d}{dt} p(t,x(t);h_\eps) \\[0.2cm]
		&=& \ds p_t(t,x(t);h_\eps) + p_{x}(t,x(t);h_\eps) \dot{x}(t) \\[0.1cm]
		&=& \ds p_t(t,x(t);h_\eps) + Dh_\eps(x(t))  \dot{x}(t).
	\end{array}
\]
which is equivalent to
\[
	\begin{array}{rcl}
	 	\ds p_t(t,x(t);h_\eps) &=&\ds \dot{y}(t) - Dh_\eps(x(t))  \dot{x}(t) \\[0.1cm]
		&=& \ds\frac{1}{\eps}g(x(t),y(t)) - Dh_\eps(x(t)) f(x(t),y(t))
	\end{array}	 
\]		
By using $h_\eps(x(t)) = y(t)$ and the invariance equation \eqref{eq:invariance}, we get
\[
	p_t =  \frac{\partial}{\partial t} p = 0 \qquad \LRa\qquad  \begin{pmatrix} \frac{\partial}{\partial t} p^{(1)} \\ \vdots \\  \frac{\partial}{\partial t} p^{(n-k)} \end{pmatrix} = 0.
\]}
We have to prove, that $p_t = 0$ implies $K(\sigma_i) = 0$ for all ${q} \in \calM$ and $i=2,...,k+1.$ The metric tensor is given by
\[
	{\g = \begin{pmatrix}
				1+  p^{(i)}_tp^{(i)}_t & p^{(i)}_tp^{(i)}_{x_1} & \dots & p^{(i)}_tp^{(i)}_{x_k} \\
				p^{(i)}_{x_1}p^{(i)}_t & 1+  p^{(i)}_{x_1}p^{(i)}_{x_1}  & \dots & p^{(i)}_{x_1}p^{(i)}_{x_k} \\
				          \vdots					&										&	\ddots & \vdots \\
				p^{(i)}_{x_k}p^{(i)}_t & p^{(i)}_{x_k}p^{(i)}_{x_1}  & \dots & 1+p^{(i)}_{x_k}p^{(i)}_{x_k} \\
		  \end{pmatrix}},
\]
denoting the partial derivatives with a subscript, i.e.\ ${p^{(i)}_{x_j} := \frac{\partial}{\partial x_j}p^{(i)}}.$ Because of $p_t^{(i)} = 0$ for all ${q} \in \calM$ and $i=1,...,n-k,$ we have
\[
	{\g} = 	\begin{pmatrix}
					1 & 0 \\
					0  & P
			\end{pmatrix}
\]
with $P \in \R^{k\times k}.$ The contravariant metric tensor is given by
\[
	{\g^{-1}} = \begin{pmatrix}
					1 & 0 \\
					0  & P^{-1}
			\end{pmatrix}.
\]
Note that since {$\g$} is symmetric and positive definite, all principal minors are positive and thus $P$ is regular. Further, it holds 
\[
	K(\sigma_i) = \frac{R^{m}_{1ii} g_{m1}}{g_{11}g_{ii}-g_{1i}^2} = \frac{R^{1}_{1ii} g_{11}}{g_{11}g_{ii}-g_{1i}^2} = \frac{R^{1}_{1ii}}{g_{ii}},
\]
because $g_{1m} = 0$ for $m = 2,...,k+1.$ The remaining task is to show, that it holds both $R^{1}_{1ii} = 0$ and $g_{ii} \not= 0$ for all $i = 2,...,k+1.$ The latter issue is trivial, since
$g_{ii} = 1 + \sum_{\ell=1}^{n-k}{(p^{(\ell)}_{x_{i-1}})^2} > 0.$ Moreover, the entries of the Riemann tensor are given by
\[
	R^{1}_{1ii} = \frac{\partial}{\partial x^1} \Gamma^{1}_{ii} -  \frac{\partial}{\partial x^i} \Gamma^{1}_{1i} +  \Gamma^{\ell}_{ii}\Gamma^{1}_{1\ell} - \Gamma^{\ell}_{1i}\Gamma^{1}_{i\ell}
\]
and while using $g_{1m} = 0$ for $m = 2,...,k+1$ once again the Christoffel symbols read 
\[
	\begin{array}{rcl}
		\ds \Gamma^{1}_{ij} &=&  \ds \frac{1}{2} g^{1\ell} \left(\frac{\partial}{\partial x^i} g_{j\ell} + \frac{\partial}{\partial x^j} g_{i\ell} - \frac{\partial}{\partial x^\ell} g_{ij}\right) \\
				&=&\ds \frac{1}{2} g^{11} \left(\frac{\partial}{\partial x^i} g_{j1} + \frac{\partial}{\partial x^j} g_{i1} - \frac{\partial}{\partial x^1} g_{ij}\right) \\
				&=&\ds -\frac{1}{2} \frac{\partial}{\partial x^1} g_{ij}.
	\end{array}
\]
Based on $p_t^{(i)} = 0$ for all $i = 1,...,n-k,$ we conclude 
\[
	\frac{\partial}{\partial t} {\g} = \frac{\partial}{\partial x^1} \begin{pmatrix}
					1 & 0 \\
					0  & P
			\end{pmatrix}
			= \begin{pmatrix}
					0 & 0 \\
					0  & \frac{\partial}{\partial x^1}P
			\end{pmatrix}
			= \begin{pmatrix}
					0 & 0 \\
					0  & 0
			\end{pmatrix}
\]
because it holds
\[
	\frac{\partial}{\partial x^1} P_{ij} = \frac{\partial}{\partial t} P_{ij} = \frac{\partial}{\partial t}  \sum_{\ell=1}^{n-k} {p^{(\ell)}_{x_{i}}p^{(\ell)}_{x_{j}}} = \sum_{\ell=1}^{n-k} {( \underbrace{p^{(\ell)}_{x_{i}t}}_{=0}p^{(\ell)}_{x_{j}} + p^{(\ell)}_{x_{i}} \underbrace{p^{(\ell)}_{x_{j}t}}_{=0})}  = 0
\]
using the commutation of second derivatives. Therefore, we have $\Gamma^{1}_{ij} = 0$ for all $i,j=1,...,k+1$ and thus $R^{1}_{1ii} = 0.$ 
\end{proof}
\begin{corollary}[Gaussian Curvature] \label{cor:gaussian} A special case of \cref{thm:neccond} for $(1,1)$ systems leads to the following statement: If $a(\cdot) = h_\eps(\cdot)$ holds, then the gaussian curvature vanishes for all ${q} \in \calM.$
\end{corollary}

%%%%%%%%%%%%%%%%%%%%%%%%%%%%%%%%%%%%%%%%%%%%%%
% Section: Examples
%%%%%%%%%%%%%%%%%%%%%%%%%%%%%%%%%%%%%%%%%%%%%%
\section{Examples and Results}
\label{sec:examples}
In this section, we discuss a few examples and apply the theoretical results of \Cref{sec:neccond} to demonstrate how the immersion defined in \cref{eq:immersion} can be used in applications. We discuss well known examples frequently used as test models for SIAM computation and develop similar higher-dimensional examples for illustration purposes. Further examples can be found in \cite{Heiter2017}.

%%% (1,1)-Davis-Skodje. Source: DavisSkodje1999
\subsection{(1,1)-Davis-Skodje Model}
\label{sec:sub:ds}
The Davis-Skodje model (cf.\ \cite{DavisSkodje1999}) is widely used for analysis and performance tests for manifold based model reduction techniques identifying slow invariant manifolds. The system reads
\begin{equation}
	\begin{array}{rcl}
		\ds \frac{d}{dt}x(t) &=& \ds -x(t) =: f(x(t),y(t))\\[0.25cm]
		\ds \frac{d}{dt}y(t) &=& \ds -\gamma y(t) + \frac{(\gamma -1) x(t) + \gamma x(t)^2}{(1+x(t))^2} =: g(x(t),y(t))
	\end{array}
\label{eq:ds}	
\end{equation}
with $\gamma := \eps^{-1} > 1, \eps > 0$ being the time scale separation parameter. The solution of this system is
\begin{eqnarray*}
	x(t) &=& c_1 e^{-t} \\[0.25cm]
	y(t) &=& c_2 e^{-\gamma t} + \frac{c_1}{c_1 + e^t}	
\end{eqnarray*}
with constants $c_1, c_2 \in \R$ depending on an initial value. Further, the slow invariant manifold is analytically known
\[
	S_\eps = \left\{ (x,y) \in \R^2 :  y = h_\eps(x) = \frac{x}{x+1} \right\} = S_0
\]
and identical with the critical manifold. The parameterization $p$ is given by
\[
	y = {p(t,x;a)} = a(xe^t) e^{-\gamma t} - \frac{xe^{(1-\gamma)t}}{xe^t +1} + \frac{x}{x + 1}.
\]
while using {$x$ as slow variable and $y$ as fast variable}. \Cref{fig:spacetime_ds} depicts the parameterization $p(t,x;a)$ for $a(u) = 1-\frac{1}{2}u$ (left) and $a(u) = (u-1)^2$ (right) for $\gamma = 3.5$ and $(t,x) \in [0,2]^2.$ 
Based on \cref{cor:gaussian}, we focus on the calculation of the gaussian curvature via the Riemann tensor identifying $(x^1,x^2,x^3) := (t,x,y).$ 
We omit the explicit dependencies due to simplicity. 
The induced metric tensor of \cref{cor:metricTensor} is
\[
	{\g} = 		\begin{pmatrix} 
							1 + p^2_t & p_tp_{x} \\ p_tp_{x} &  1 + p^2_{x}
			\end{pmatrix}
\]
with
\[
	p_t =  \frac{d}{dt} {p(t,x;a)} = a'(xe^t)x e^{(1-\gamma) t}-a(x e^t) \gamma e^{-\gamma t}- \frac{x (1-\gamma) e^{(1-\gamma) t}}{x e^t+1}+ \frac{x^2 e^{(2-\gamma) t}}{(x e^t+1)^2}
\]
and
\[	
	p_{x} = \frac{d}{dx} {p(t,x;a)} = a'(x e^t)e^{(1-\gamma) t}- \frac{e^{(1-\gamma) t}}{x e^t+1}+ \frac{x e^{(2-\gamma) t}}{(x e^t+1)^2}+\frac{1}{(x+1)}- \frac{x}{(x+1)^2}.
\]
Then, it holds for the time sectional curvature
\[
	K(\sigma_2) = \frac{R^{m}_{122} g_{m1}}{g_{11}g_{22}-g_{12}^2} = \frac{R^{1}_{122} g_{11} + R^{2}_{122} g_{12}}{g_{11}g_{22}-g_{12}^2} = \frac{p_{tt} p_{xx} - p_{xt}^2}{(1+p_t^2 + p_{x}^2)^2}.
\]
We investigate which initial value function $a$ cancels the time derivative of $p.$ Therefore, we try to find a function $a$ such that
\[
	p_t =  a'(xe^t)x e^{(1-\gamma) t}-a(x e^t) \gamma e^{-\gamma t}- \frac{x (1-\gamma) e^{(1-\gamma) t}}{x e^t+1}+ \frac{x^2 e^{(2-\gamma) t}}{(x e^t+1)^2} \stackrel{!}{=} 0.
\]
Substituting $u := x e^t \Ra x = u e^{-t}$ and multiplying with $e^{\gamma t}$ leads to the differential equation 
\[
	a'(u) u - \gamma a(u) + \frac{(\gamma-1) u}{u +1} + \frac{u^2}{(u+1)^2} = 0
\]
whose solution
\[
	a(u) = \frac{u}{u+1} + c u^\gamma \qquad ,\;c \in \R
\]
describes invariant graphs under the flow. Note that the constant $c$ is arbitrary. Such solutions annul $p_t$ and therefore the gaussian curvature $K(\sigma_2) = 0.$ Especially, we obtain $a = h_\eps$ by setting $c = 0.$ 

%%% (1,1)-Nonlinear Model. Source: Kueh2015
\subsection{(1,1)-Nonlinear Model}
\label{sec:sub:kuehn_nlm}
The following example is taken from \cite{Kuehn2015} (Example 3.1.6). Consider the (1,1) system 
\[
	\begin{array}{rcl}
		\ds \frac{d}{dt}x(t) &=& \ds - \eps x(t)  \\[0.25cm]
		\ds \frac{d}{dt}y(t) &=& \ds x^2(t) - y(t)
	\end{array}
\]
with {$\eps > 0$ having} the solution
\begin{eqnarray*}
	x(t) &=& c_1 e^{-\eps t} \\[0.25cm]
	y(t) &=& \left(c_2 - \frac{c^2_1}{1-2\eps}\right) e^{-t} +  \frac{c^2_1}{1-2\eps} e^{-2\eps t}
\end{eqnarray*}
with constants $c_1,c_2 \in \R$ depending on an initial value. In contrast to the (1,1)-Davis-Skodje model, the critical manifold
\[
	S_0 = \left\{ (x,y) \in \R^2 :  y = h_0(x) = x^2 \right\}
\]
is not equal to the slow invariant manifold
\[
	S_\eps = \left\{ (x,y) \in \R^2 :  y = h_\eps(x) = \frac{x^2}{1-2\eps} \right\}
\]
enabling to check the necessary condition of \cref{thm:neccond} with both $h_0$ and $ h_\eps.$ The required parameterization $p$ is given by
\[
	y = {p(t,x;a)} = a(xe^{\eps t}) e^{-t} - \frac{x^2 e^{(2\eps -1)t}}{1-2\eps} + \frac{x^2}{1-2\eps}.
\]
The induced metric tensor, the Riemann tensor and the time sectional curvature $K(\sigma_2)$ can be calculated similar to \Cref{sec:sub:ds}. Consider the equation
\[
	p_t = a'(x e^{\eps t}) x \eps e^{(\eps-1) t} - a(x e^{\eps t}) e^{-t} + x^2 e^{(2 \eps-1) t} \stackrel{!}{=} 0
\]
which identifies the initial value functions annulling the gaussian curvature. 
Substituting $u := x e^{\eps t} \Ra x = u e^{-\eps t}$ and multiplying with $e^{t}$ leads to the differential equation
\[
	a'(u) u \eps - a(u) + u^2 = 0.
\]
The solution is
\[
	a(u) = \frac{u^2}{1-2\eps} + c u^{\eps^{-1}}
\]
with an arbitrary constant $c \in \R.$ Therefore, we obtain $a = h_\eps$ by setting $c = 0$ and further it holds $K(\sigma_2) = 0$ for all points on $\calM.$ Notice, that $h_0$ does not have such a structure and the gaussian curvature does not vanish for all points of $\calM$. In particular for 
\[
	(t^*,x^*,{p(t^*,x^*;h_0)}) \in \calM = \{ (t,x,y) \in \R^3 \;:\; {y = p(t,x;h_0)} \}
\]
with $\eps = 0.01, t^* = 0, x^* = 0.5$ it holds 
\[
	K(\sigma_2)(0,0.5,0.25) \approx -0.00255 \neq 0.
\]
Hence, the necessary condition of \cref{thm:neccond} is not fulfilled for the critical manifold in this case.

%%% (1,1)-Nonlinear Model. Source: Kueh2015
\subsection{(1,1)-Enzym-Kinetic Model}
The Michaelis-Menten-Henri model for enzym kinetics reads
\begin{eqnarray*}
	\frac{d}{dt}x(t) &=& \eps \big(-x(t) + (x(t) + \kappa - \lambda) y(t)\big) \\[0.25cm]
	\frac{d}{dt}y(t) &=& x(t) - (x(t) + \kappa)y(t),
\end{eqnarray*}
cf.\ \cite{Kuehn2015}, Example 11.2.4, assuming $\kappa > \lambda > 0$ and $(x(t),y(t)) \in \R^+ \times \R^+$ for all $t.$ In contrast to the presented examples in \Cref{sec:sub:ds} and \Cref{sec:sub:kuehn_nlm}, the differential equation is not analytically solvable anymore. Further, the critical manifold
\[
	S_0 = \left\{ (x,y) \in \R^2 :  y = h_0(x) = \frac{x}{x+\kappa} \right\}
\]
does not coincide with the slow invariant manifold also in this example. The graph of the slow invariant manifold can be computed via an asymptotic expansion for $h_\eps.$ Thus, we have
\[
	S_\eps = \left\{ (x,y) \in \R^2 :  y = h_\eps(x) = \sum_{k=0}^\infty  \eps^k h_k(x)\right\}
\]
and the coefficients $h_k$ can be determined using the invariance equation \cref{eq:invariance}. 

Since an analytical solution for the system is not known, the parameterization $p$ can only be obtained by solving the boundary value problem \cref{eq:bvp} {which in this case reads} for
\begin{equation*}
	\text{(BVP)} \begin{cases}	\dot{x}(t) = \eps \big(-x(t) + (x(t) + \kappa - \lambda) y(t)\big) \\
					\dot{y}(t) = x(t) - (x(t) + \kappa)y(t) \\
					x(t^*) = x^* \\
					y(0) = a(x(0))
	\end{cases}
\end{equation*}
$t \in [0,t^*]$ and depending on an initial value function $a$ and a fixed point $(t^*,x^*).$ {We solve those problems with \textsc{Matlab}'s \texttt{bvp4c} method.} The solution $y^* = y(t^*)$ of this boundary value problem provides the required parameterization
\[
	y^* = {p(t^*,x^*;a)} = \arg \text{BVP}(y(t^*)).
\]

Thereby, the induced metric tensor, the Riemann tensor and the time sectional curvature $K(\sigma_2)$ can be calculated by utilizing central differences to approximate the required derivatives. 
\cref{fig:enzymGaussian} depicts the gaussian curvature $K(\sigma_2){(t^*,x^*)}$ (colored surface) on a reference domain $\Omega = [0,2] \times [0,3]$ for $a(u) = 0.5$ (red thick line). 
\begin{figure}[tbhp]
	\centering
	\includegraphics[scale=0.50]{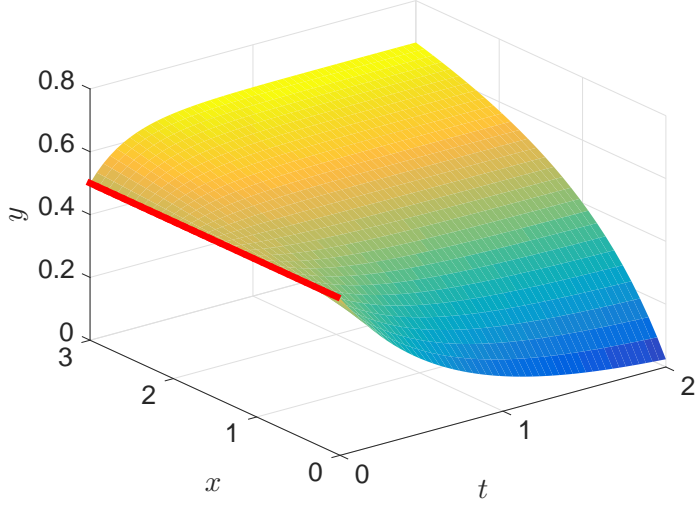}
	\caption{Illustration of an (1,1)-enzym kinetic model in space-time with $\lambda = 0.5, \kappa = 1.5$ and $\eps = 0.01$. Colored Surface: $\calM$. Red Curve: Graph of initial value functions $a(u) = 0.5$ {projected into the $xy$-plane}. The BVPs for each point are solved with \textsc{Matlab}'s \texttt{bvp4c} method.}
	\label{fig:enzymGaussian}
\end{figure}
Let be $\Omega_h$ a discretization of $\Omega$. In order to demonstrate the necessary condition for this example, we calculate 
\[
	\calI[a] = \int_\Omega |K(\sigma_2)(t,x) | \;d(t,x) \approx \frac{1}{| \Omega_h | } \sum_{(t_h,{x_h}) \in \Omega_h}  |K(\sigma_2)(t_h,x_h) |
\]
for several initial value functions 
\[
	a_k(u) := \sum_{\ell=0}^k \eps^\ell h_\ell(u)
\]
approximating $h_\eps.$ We choose  $\lambda = 0.5, \kappa = 1.0, \eps = 0.5$ and a grid with equidistant discretized axis 
\[
	\Omega_h = \{t_1,t_2,...,t_{n_t} \} \times \{x_1,x_2,...,x_{n_x} \}
\]
with $n_t = 10$ and $n_x = 20.$

 \cref{fig:enzymEpsSeries} illustrates for increasing $k$ the decrease of $\calI[a_k]$ validating the necessary condition of \cref{thm:neccond}. It must hold $\calI[a_k] \to 0$ for $k \to \infty.$

\begin{figure}[tbhp]
	\centering
	\includegraphics[scale=0.23]{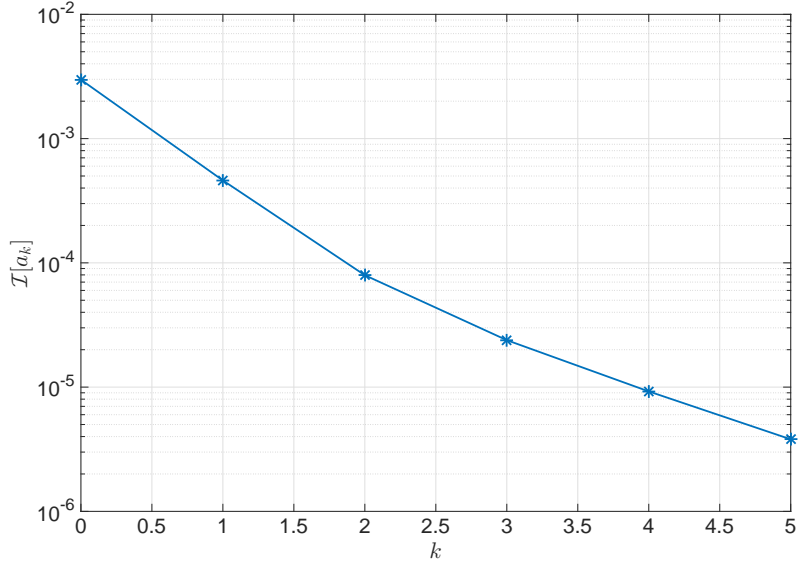}
	\caption{Validation of the necessary condition by means of the (1-1)-enzym kinetic model. The decrease of  $\calI[a_k]$ is plotted for $k = 0,...,5$ having  $\lambda = 0.5, \kappa = 1.0$ and $\eps = 0.5$ as model parameters.}
	\label{fig:enzymEpsSeries}
\end{figure}

%%% (2,1)-Higher Dimensional Model. Based on DS
\subsection{(2,1)-Higher-dimensional Model I}
We discuss a higher dimensional (2,1) system to demonstrate the application of \cref{thm:neccond} to a model with two slow and one fast variable in the case, that the critical manifold coincides with the slow invariant manifold. Consider a (2,1) system which is based on the Davis-Skodje model and extended with another slow variable. The system reads
\[
	\begin{array}{rcl}
		\ds \frac{d}{dt}x_1(t) &=& \ds - x_1(t) \\
		\ds \frac{d}{dt}x_2(t) &=& \ds - 2x_2(t) \\
		\ds \frac{d}{dt}y(t) &=& \ds  -\gamma y(t) + \frac{(\gamma-1)x_1(t)+ \gamma x^2_1(t)}{(1+x_1(t))^2} + \frac{2(\gamma-2)x_2(t) + 2\gamma x^2_2(t)}{(1+x_2(t))^2}
	\end{array}
\]
with $\gamma := \eps^{-1} > 2, \eps > 0$ being the time scale separation parameter. The system has the following solution
\begin{eqnarray*}
	x_1(t) &=& c_1 e^{-t} \\[0.25cm]
	x_2(t) &=& c_2 e^{-2t}  \\[0.25cm]	
	y(t) &=& c_3 e^{-\gamma t} + \frac{c_1}{c_1+1} + \frac{2c_2}{c_2+1}
\end{eqnarray*}
with constants $c_1,c_2,c_3 \in \R$ depending on an initial value. In this example, the slow invariant manifold 
\[
	S_\eps = \left\{ (x_1,x_2,y) \in \R^3 \;:\; y = h_\eps(x_1,x_2) = \frac{x_1 + 2x_2 + 3 x_1x_2}{(1+x_1)(1+x_2)} \right\} = S_0
\]
coincides with the critical manifold as they do in the Davis-Skodje model. The required parameterization $p$ is given by
\[
	y = {p(t,x_1,x_2;a)} = \left( a(x_1e^{t},x_2e^{2t})  - \frac{x_1e^{t}}{1+x_1e^{t}} - \frac{2x_2e^{2t}}{1+x_2e^{2t}}\right) e^{-\gamma t} + \frac{x_1}{1 + x_1} + \frac{2 x_2}{1+x_2} .
\]
The initial value function $a = a(u_1,u_2)$ depends on two variables. We identify $(x^1,x^2,x^3,x^4) = (t,x_1,x_2,y)$ and omit the explicit dependencies because of simplicity, once again. The induced metric tensor is
\[
	{\g} = 		\begin{pmatrix} 
							1 + p^2_t & p_tp_{x_1} & p_tp_{x_2} \\ p_tp_{x_1} &  1 + p^2_{x_1} & p_{x_1}p_{x_2} \\ p_{t}p_{x_2} & p_{x_1}p_{x_2} & 1 + p^2_{x_2}							
			\end{pmatrix}
\]
and the time sectional curvatures are given by
\begin{eqnarray*}
	K(\sigma_2) &=& \frac{R^{m}_{122} g_{m1}}{g_{11}g_{22}-g_{12}^2} = \frac{R^{1}_{122} g_{11} + R^{2}_{122} g_{12} + R^{3}_{122} g_{13}}{g_{11}g_{22}-g_{12}^2}  \\
	K(\sigma_3) &=& \frac{R^{m}_{133} g_{m1}}{g_{11}g_{33}-g_{13}^2} = \frac{R^{1}_{133} g_{11} + R^{2}_{133} g_{12} + R^{3}_{133} g_{13}}{g_{11}g_{33}-g_{13}^2}.
\end{eqnarray*}
We investigate the partial differential equation
\begin{eqnarray*}
	p_t  &=&  \left(\partial_1 a(x_1e^t,x_2e^{2t}) x_1e^t  + 2 \partial_2 a(x_1e^t,x_2e^{2t}) x_2e^{2t} - \frac{x_1 e^t}{x_1 e^t + 1} \right. \\
	&& \left.+ \frac{x_1^2 e^{2t}}{(x_1 e^t+1)^2} - \frac{4x_2e^{2t}}{x_2e^{2t}+1} + \frac{4 x_2^2 (e^{4 t}}{(x_2 e^{2 t}+1)^2}\right) e^{-\gamma t} \\
	&& - \left(  a(x_1e^t,x_2e^{2t}) - \frac{x_1e^{t}}{1+x_1e^{t}} - \frac{2x_2e^{2t}}{1+x_2e^{2t}} \right)\gamma e^{-\gamma t} = 0
\end{eqnarray*}
in order to identify the time invariants. We denote $\partial_k a$ for the partial derivative with respect to the $k$-th argument. Substituting $u_1 := x_1 e^{t}, u_2 := x_2 e^{2t}$ and multiplying with $e^{\gamma t}$ leads to
\begin{eqnarray*}
	0 &=& \partial_1a(u_1,u_2) u_1 + 2 \partial_2 a(u_1,u_2) u_2 - \frac{u_1}{(u_1 + 1)^2} - \frac{4u_2}{(u_2+1)^2}  \\ 
	&&	- \gamma \left(a(u_1,u_2) - \frac{u_1}{u_1+1} - \frac{2u_2}{u_2 +1} \right).
\end{eqnarray*}
The solution is
\[
	a(u_1,u_2) = \frac{u_1 + 2u_2 + 3 u_1u_2}{(1+u_1)(1+u_2)}  + v\left(\frac{u_2}{u^2_1}\right) u_1^{\gamma}
\]
with an arbitrary, sufficiently differentiable function $v(\cdot).$ We get $a = h_\eps$ by setting $v(\cdot) \equiv 0$ resulting in $K(\sigma_2) = 0$ and $K(\sigma_3) = 0$ for all points on $\calM.$

%%% (3,2)-Higher Dimensional Model. Based on Enzym
\subsection{(3,2)-Higher-dimensional Model II} The last example discusses a high\-er-dimensional model with three slow and two fast variables and is designed to illustrate the necessary condition in higher dimension, where the critical manifold does not coincide with the slow invariant manifold. The (3,2) system reads
\[
	\begin{array}{rcl}
		\ds \frac{d}{dt}x_1(t) &=& \ds - \eps x_1(t) \\[0.25cm]
		\ds \frac{d}{dt}x_2(t) &=& \ds - 2\eps x_2(t) \\[0.25cm]
		\ds \frac{d}{dt}x_3(t) &=& \ds - 3\eps x_3(t) \\[0.25cm]
		\ds \frac{d}{dt}y_1(t) &=& \ds  2 x^2_1(t) + x^2_2(t)x_3(t) - 4 y_1(t) \\[0.25cm]
		\ds \frac{d}{dt}y_2(t) &=& \ds  3 x^4_1(t) + x^3_2(t) - 3 y_2(t) \\		
	\end{array}
\]
with $0 < \eps \ll \frac{1}{2}.$ The analytical solution of this system is given by
\begin{eqnarray*}
	x_1(t) &=& c_1 e^{-\eps t} \\[0.25cm]
	x_2(t) &=& c_2 e^{-2 \eps t}  \\[0.25cm]	
	x_3(t) &=& c_3 e^{-3 \eps t}  \\[0.25cm]
	y_1(t) &=& c_4 e^{-4 t} + \frac{c^2_1}{2-\eps} e^{-2 \eps t} + \frac{c_2^2 c_3}{4-7\eps} e^{-7\eps t} \\[0.25cm]
	y_2(t) &=& c_5 e^{-3 t} + \frac{3c^4_1}{3-4\eps} e^{-4 \eps t} + \frac{c_2^3}{3-6\eps} e^{-6\eps t}
\end{eqnarray*}
with constants $c_1,c_2,c_3,c_4,c_5 \in \R$ depending on an initial value. The critical manifold 
\[
	S_0 = \left\{ (x_1,x_2,x_3,y_1,y_2) \in \R^5 \;:\; \begin{pmatrix} y_1 \\y_2 \end{pmatrix} = h_0(x_1,x_2,x_3) = \begin{pmatrix} \frac{1}{2}x^2_1 + \frac{1}{4}x_2^2 x_3 \\[0.15cm]  x^4_1 + \frac{1}{3}x_2^3 \end{pmatrix} \right\}
\]
does not coincide with the slow invariant manifold 
\[
	S_\eps = \left\{ (x_1,x_2,x_3,y_1,y_2) \in \R^5 \;:\; \begin{pmatrix} y_1 \\y_2 \end{pmatrix} = h_\eps(x_1,x_2,x_3) = \begin{pmatrix} \frac{x^2_1}{2-\eps} + \frac{x_2^2 x_3}{4-7\eps} \\[0.15cm]   \frac{3x^4_1}{3-4\eps} + \frac{x_2^3}{3-6\eps} \end{pmatrix} \right\}.
\]
The required parameterization $p$ reads
\[
	\begin{pmatrix}
		y_1 \\ y_2
	\end{pmatrix}	
	= 
	{p(t,x_1,x_2,x_3;a) = \begin{pmatrix}
		p^{(1)}(t,x_1,x_2,x_3;a) \\
		p^{(2)}(t,x_1,x_2,x_3;a)
	\end{pmatrix}}
\]
with
\begin{eqnarray*}
	{p^{(1)}(t,x_1,x_2,x_3;a)} &=& \left(a^{(1)}(x_1e^{\eps t},x_2e^{2\eps t},x_3e^{3\eps t}) - \frac{x_1^2 e^{2\eps t}}{ 2-\eps} - \frac{x_2^2x_3 e^{7\eps t}}{ 4-7\eps} \right) e^{-4t} \\
		&& + \frac{x_1^2}{2-\eps} + \frac{x_2^2 x_3}{4-7\eps} \\
	{p^{(2)}(t,x_1,x_2,x_3;a)} &=& \left(a^{(2)}(x_1e^{\eps t},x_2e^{2\eps t},x_3e^{3\eps t}) - \frac{3x_1^4 e^{4\eps t}}{ 3-4\eps} - \frac{x_2^3 e^{6\eps t}}{ 3-6\eps} \right) e^{-3t} \\
		&& + \frac{3x_1^4}{3-4\eps} + \frac{x_2^3}{3-6\eps}.
\end{eqnarray*}
We identify $(x^1,x^2,x^3,x^4,x^5,x^6) = (t,x_1,x_2,x_3,y_1,y_2)$ and omit the explicit dependencies due to simplicity. The induced metric tensor is
{\small\[
	{\g} = 		\begin{pmatrix} 
							1 + \sum_{\ell=1}^2 (p^{(\ell)}_{t})^2  & p^{(1)}_tp^{(1)}_{x_1} + p^{(2)}_tp^{(2)}_{x_1} & p^{(1)}_tp^{(1)}_{x_2} + p^{(2)}_tp^{(2)}_{x_2}  & p^{(1)}_tp^{(1)}_{x_3} + p^{(2)}_tp^{(2)}_{x_3} \\[0.15cm]
							p^{(1)}_tp^{(1)}_{x_1} + p^{(2)}_tp^{(2)}_{x_1} &  1 + \sum_{\ell=1}^2 (p^{(\ell)}_{x_1})^2   & p^{(1)}_{x_1}p^{(1)}_{x_2} + p^{(2)}_{x_1}p^{(2)}_{x_2} & p^{(1)}_{x_1}p^{(1)}_{x_3} + p^{(2)}_{x_1}p^{(2)}_{x_3}\\[0.15cm]
							p^{(1)}_tp^{(1)}_{x_2} + p^{(2)}_tp^{(2)}_{x_2} & p^{(1)}_{x_1}p^{(1)}_{x_2} + p^{(2)}_{x_1}p^{(2)}_{x_2}&  1 + \sum_{\ell=1}^2 (p^{(\ell)}_{x_2})^2  & p^{(1)}_{x_2}p^{(1)}_{x_3} + p^{(2)}_{x_2}p^{(2)}_{x_3} \\[0.15cm]
							p^{(1)}_tp^{(1)}_{x_3} + p^{(2)}_tp^{(2)}_{x_3} & p^{(1)}_{x_1}p^{(1)}_{x_3} + p^{(2)}_{x_1}p^{(2)}_{x_3} & p^{(1)}_{x_2}p^{(1)}_{x_3} + p^{(2)}_{x_2}p^{(2)}_{x_3} &  1 + \sum_{\ell=1}^2 (p^{(\ell)}_{x_3})^2 
			\end{pmatrix}
\]}
and the time sectional curvatures are given by
\begin{eqnarray*}
	K(\sigma_2) &=& \frac{R^{m}_{122} g_{m1}}{g_{11}g_{22}-g_{12}^2} = \frac{R^{1}_{122} g_{11} + R^{2}_{122} g_{12} + R^{3}_{122} g_{13} + R^{4}_{122} g_{14}}{g_{11}g_{22}-g_{12}^2}  \\
	K(\sigma_3) &=& \frac{R^{m}_{133} g_{m1}}{g_{11}g_{33}-g_{13}^2} = \frac{R^{1}_{133} g_{11} + R^{2}_{133} g_{12} + R^{3}_{133} g_{13} + R^{4}_{133} g_{14}}{g_{11}g_{33}-g_{13}^2} \\
	K(\sigma_4) &=& \frac{R^{m}_{144} g_{m1}}{g_{11}g_{44}-g_{14}^2} = \frac{R^{1}_{144} g_{11} + R^{2}_{144} g_{12} + R^{3}_{144} g_{13} + R^{4}_{144} g_{14}}{g_{11}g_{44}-g_{14}^2}.
\end{eqnarray*}
We investigate the system of partial differential equations 
\begin{eqnarray*}
	p^{(1)}_t &=& 0 \\
	p^{(2)}_t &=& 0.
\end{eqnarray*}
Substitution of $u_1 := x_1e^{\eps t}, u_2 := x_2e^{2\eps t}, u_3 := x_3e^{3\eps t}$ and multiplication of the first equation with $e^{4t}$ and the second equation with $e^{3t}$ leads to the following system
\begin{eqnarray*}
	0 &=& \partial_1 a^{(1)}(u_1,u_2,u_3) u_1 \eps + 2 \partial_2 a^{(1)}(u_1,u_2,u_3) u_2 \eps + 3\partial_3 a^{(1)}(u_1,u_2,u_3) u_3 \eps \\
		&& + \frac{2u_1^2 \eps}{\eps -2} + \frac{7 u_2^2 u_3 \eps}{7\eps -4}  - 4 \left( a^{(1)}(u_1,u_2,u_3) - \frac{u_1^2}{2-\eps} - \frac{u_2^2 u_3}{4-7\eps} \right)\\[0.2cm]
	0 &=& \partial_1 a^{(2)}(u_1,u_2,u_3) u_1 \eps + 2 \partial_2 a^{(2)}(u_1,u_2,u_3) u_2 \eps + 3\partial_3 a^{(2)}(u_1,u_2,u_3) u_3 \eps \\
		&& + \frac{12u_1^4 \eps}{4\eps -3} + \frac{2 u_2^3  \eps}{2\eps -1}  - 3 \left( a^{(2)}(u_1,u_2,u_3) - \frac{3u_1^4}{3-4\eps} - \frac{u_2^3}{3-6\eps} \right).
\end{eqnarray*}
The solution is given by 
\begin{eqnarray*}
	a^{(1)}(u_1,u_2,u_3) &=& \frac{u_1^2}{2-\eps} + \frac{u_2^2 u_3}{4-7\eps} + v_1\left(\frac{u_2}{u_1^2},\frac{u_3}{u_1^3} \right) u_1^{4\eps^{-1}} \\
	a^{(2)}(u_1,u_2,u_3) &=& \frac{3u_1^4}{3-4\eps} + \frac{u_2^3}{3-6\eps} + v_2\left(\frac{u_2}{u_1^2},\frac{u_3}{u_1^3} \right) u_1^{3\eps^{-1}}
\end{eqnarray*}
with two arbitrary functions $v_1(\cdot,\cdot)$ and $v_2(\cdot,\cdot).$ Once we set $v_1 \equiv 0$ and $v_2 \equiv 0,$ it holds $a = h_\eps$ and therefore
\[
	K(\sigma_2) = 0, \;K(\sigma_3) = 0 \quad \text{and} \quad K(\sigma_4) = 0
\]
for all points on $\calM.$ 
Notice that $h_0$ does not have such a structure and the gaussian curvature does not vanish for all points of $\calM$. 
In particular for 
\[
	(t^*,x_1^*,x_2^*,x_3^*,{p^{(1)}(t^*,x_1^*,x_2^*,x_3^*;h_0),p^{(2)}(t^*,x_1^*,x_2^*,x_3^*;h_0)}) \in \calM 
\]
with $\eps = 0.01, t^* = 0, x_1^* = 0.3, x_2^* = 1, x_3^* = 0.5,$ it holds 
\begin{eqnarray*}
	K(\sigma_2)(0,0.3,1,0.5) &\approx& -0.04928 \neq 0, \\
	K(\sigma_3)(0,0.3,1,0.5) &\approx& -0.02973 \neq 0, \\
	K(\sigma_4)(0,0.3,1,0.5) &\approx& -0.01270 \neq 0.		
\end{eqnarray*}
Hence, this is a higher dimensional example where the necessary condition of \cref{thm:neccond} is not fulfilled for the critical manifold.

%%%%%%%%%%%%%%%%%%%%%%%%%%%%%%%%%%%%%%%%%%%%%%
% Section: Investigation on Sufficient Conditions
%%%%%%%%%%%%%%%%%%%%%%%%%%%%%%%%%%%%%%%%%%%%%%
\section{Investigation on Sufficient Conditions}
\label{sec:sufficient}
The statement of \cref{thm:neccond} is indeed a reformulation of the invariance equation in a differential geometry context and provides possible candidates for the slow invariant manifold or, respectively, allows testing candidate graphs by the sectional curvature criterion. 
In particular, in the case of higher-dimensional manifolds this might be a useful coordinate independent alternative to the invariance equation. 
This section deals with further investigations related to ongoing search for a sufficient differential geometric condition to characterize slow invariant manifolds of arbitrary dimension. 
The ideal case would be a pointwisely formulated additional geometric condition for the manifold graph. 
For illustration we focus on the (1,1)-Davis-Skodje model, see \Cref{sec:sub:ds}, with an analytically known slow invariant manifold.
Recall the initial value functions
\[
	a(u) = \frac{u}{u+1} + c u^\gamma
\]
with an arbitrary constant $c \in \R$ yielding zero time-sectional curvature and thus invariance under the phase flow. 
If we want to finally identify the slow invariant manifold by help of an addition condition, the latter must imply $c = 0$. 
A geometrically motivated idea discussed in our previous publications might related to minimal curvature of the graph. 
Based on the standard curvature notion, which is simple for one-dimensional manifolds, i.e.\ curves, consider the following criterion for a fixed $u$
\[
	\min_c \; F_1(c) := \min_c\; a''(u)^2 = \min_c \; \left(- \frac{2}{(u+1)^2}+\frac{2u}{(u+1)^3}+c \gamma (\gamma-1)u^{\gamma-2}\right)^2.
\]
We minimize the squared curvature in order to avoid problems with changing sign or non-differentiabilty of the absolute value function. Differentiating with respect to $c$ leads to the following result
\begin{eqnarray*}
	F_1'(c) &=& 0 \qquad \LRa \qquad c = \frac{2u^{2-\gamma}}{\gamma(\gamma u^3+3\gamma u^2-u^3+3 \gamma u-3 u^2+\gamma-3 u-1))}\\
	F_1''(c) &=& 2 \left( \gamma(\gamma-1) u^{\gamma-2}\right)^2 > 0 \quad \ \forall \, u \neq 0.
\end{eqnarray*}
For example, if we choose $\gamma = 3.5$ and $u = 2$, it yields $c \approx 0.003$, $c$ small but not zero. 
The classical curvature of the graph does not provide a sufficient condition for the slow invariant manifold. 
The following investigations are motivated from the trajectory-based optimization approach, cf.\ \cite{Lebiedz2010,Lebiedz2011,Lebiedz2016}. 
There, the slow invariant manifold is approximated by minimizing an objective functional under the constraints of the dynamics and the fixation of reaction progress variables as parametrization of the flow manifold. 
We make use of this objective functional considering the pointwise version to obtain the following additional criterion for the slow manifold graph
\[
	\min_c \; F_2(c) := \min_c\; \left\vert\left\vert J(u,a(u))\cdot \begin{pmatrix} f(u,a(u)) \\ g(u,a(u))  \end{pmatrix} \right\vert\right\vert_2^2
\]
for a fixed $u$ and the Jacobian matrix $J$ of the system \cref{eq:ds}. It holds
\[
	F_2'(c) = 0 \quad \LRa \quad c = \frac{u(u-1)}{u^\gamma\gamma^2(u^3+3u^2+3u+1)}
\]
and
\[
	F_2''(c) =2 \left(u^{\gamma} \gamma^2\right)^2 > 0  \quad \ \forall \, u \neq 0.
\]
Once again, we choose $\gamma = 3.5$ and $u = 2$ and obtain $c \approx 0.00053.$ In comparison with the first criterion, the value of $c$ is an order of magnitude smaller and closer to the SIM. 

In \cite{Lebiedz2016}, another criterion is motivated by analogy reasoning in terms of Hamilton's principle of classical mechanics. Consider
\begin{equation}
	\min_c \; F_3(c) := \min_c\;  k_1 \left\vert\left\vert \begin{pmatrix} f(u,a(u)) \\ g(u,a(u))  \end{pmatrix} \right\vert\right\vert_2^2 - k_2  \left\vert\left\vert \begin{pmatrix} u \\ a(u)  \end{pmatrix} \right\vert\right\vert_2^2.
\label{eq:sufficientDS}	
\end{equation}
The first summand can be seen as a 'generalized kinetic energy' and the second summand correspond to some appropriately defined 'generalized potential energy'. The analysis in \cite{Lebiedz2016} reveals the choice $k_1 = 1$ and $k_2 = \frac{\gamma}{u+1}$ in order to exactly identify the slow invariant manifold for the Davis-Skodje model as a time-parameterized solution of the corresponding variational problem minimizing the Lagrangian integral.

In our context here, minimizing the objective function $F_3$ reveals $c = 0$ if $k_1 = 1$ and $k_2 = \frac{\gamma}{u+1}.$ The first derivative of $F_3$ reads
\begin{eqnarray*}
	F_3'(c) &=& 2 k_1 \gamma \; \frac{\gamma u^{2 \gamma+5} c+5 \gamma u^{4+2 \gamma} c+10 \gamma u^{3+2 \gamma} c+10 \gamma u^{2+2 \gamma} c}{(u+1)^5} \\
		&& + 2 k_1\; \gamma \frac{5 \gamma u^{1+2 \gamma} c+\gamma c u^{2 \gamma}+3 u^{2+\gamma}+3 u^{\gamma+3}+u^{4+\gamma}+u^{1+\gamma}}{(u+1)^5} \\
		&& - 2 k_2 \left(\frac{u}{u+1}+c u^\gamma\right) u^\gamma.
\end{eqnarray*}
Solving $F_3'(c) = 0$ yields 
\begin{equation}
	c = - \frac{(\gamma k_1-u k_2-k_2) u}{(\gamma^2 u^2 k_1+2 \gamma^2 u k_1+\gamma^2 k_1-u^2 k_2-2 u k_2-k_2) u^\gamma}.
\label{eq:calcC}	
\end{equation}
For $c=0$ the following equation has to be satisfied 
\[
	(\gamma k_1-u k_2-k_2) u = 0
\]
using only the numerator of \cref{eq:calcC}. This leads to either $u=0$ or
\[
	k_2 = \frac{\gamma k_1}{u+1}
\]
and therefore the derived values of $k_1$ and $k_2$ coincide by setting $k_1 = 1$ with the constants derived in \cite{Lebiedz2016}. Further, it holds
\[
	F''(c) = \frac{2 \gamma (u^{\gamma})^2 (\gamma u+\gamma-1)}{u+1} \neq 0  \quad \forall \quad u \in \R \setminus \left\{0, \frac{1-\gamma}{\gamma}\right\}.
\]

Hence, the slow invariant manifold for the Davis-Skodje model can be sufficiently characterized by
\[
	a = h_\eps \quad \LRa \quad \text{$a$ satisfies \cref{thm:neccond} and $a$ minimizes problem \cref{eq:sufficientDS}.}
\]

\begin{figure}
	\includegraphics[scale=0.24]{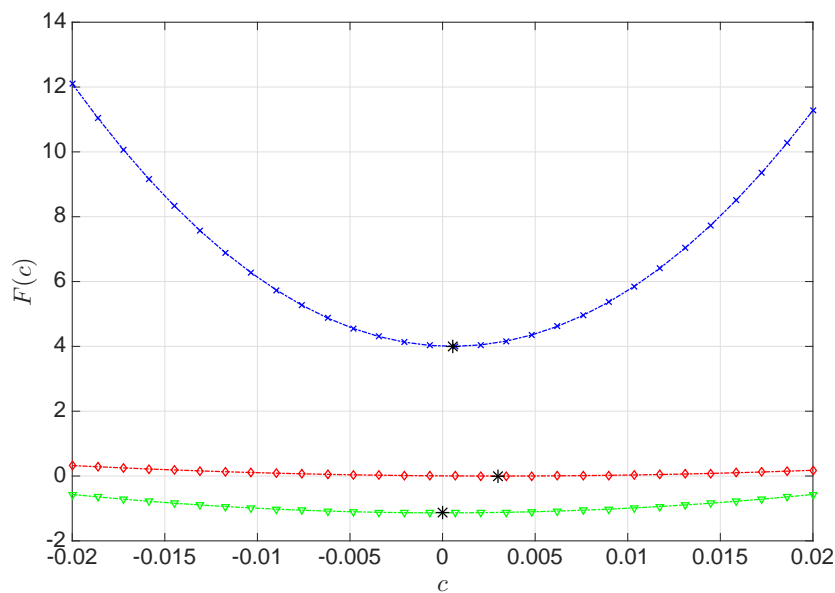}
	\caption{Illustration of each criterion function $F_1$ (red line with diamonds), $F_2$ (blue line with crosses) and $F_3$ (green line with triangles) applied to the (1,1)-Davis-Skodie model having $\gamma = 3.5$ and $u = 2.$ The black stars mark the particular minima.}
	\label{fig:sufficient}
\end{figure}

\cref{fig:sufficient} depicts the graphs of $F_1$ (red line with diamonds), $F_2$ (blue line with crosses) and $F_3$ (green line with triangles) for $c \in [-0.05,0.05].$ The black stars mark the minima.

%%%%%%%%%%%%%%%%%%%%%%%%%%%%%%%%%%%%%%%%%%%%%%
% Section: Summary and Conclusions
%%%%%%%%%%%%%%%%%%%%%%%%%%%%%%%%%%%%%%%%%%%%%%
\section{Summary and Conclusions}
\label{sec:conclusion}
The focus of this work lies on the investigation of a general differential geometric viewpoint in order to characterize slow invariant attracting manifold in multiple time scale dynamical systems. 
A necessary condition is formulated in \Cref{sec:neccond} as a reformulation of the invariance equation with differential geometric terms and is illustrated by means of several examples in \Cref{sec:examples}. 
A discussion on a sufficient condition is presented in \Cref{sec:sufficient}. 
Ideas and investigation are based upon \emph{intrinsic curvature} up to now, an issue for further investigation might be \emph{extrinsic curvature} measuring the curvature of the manifold $\calM$ as a geometric object embedded in the surrounding space. 
The work of Ginoux et al.\ \cite{Ginoux2006,Ginoux2008} uses extrinsic curvature of curves in hyperplanes (codimension-1 manifolds) in order to derive a determinant criterion for computing slow manifold points. 
A generalization to embedded manifolds of arbitrary dimension would be desirable, in the ideal case involving a local geometric criterion that can easily be evaluated numerically. For this aim we propose to consider the slow manifold as a submanifold of the full solution manifold of the ODE flow in an extended { phase}-space time frame and look for a way to define an appropriate metric on this manifold. We conjecture that the SIM could be characterized within that viewpoint by extremals or zeros of an appropriate external curvature notion. Our preliminary studies on sufficient geometric characterization of the slow manifold in simple test models investigated in this work might help for this purpose.

\section*{Acknowledgments}
The authors thank Marcus Heitel for discussions on the topic.

\bibliographystyle{plain}
\bibliography{references}

\begin{thebibliography}{10}

\bibitem{Adrover2007a}
A.~Adrover, F.~Creta, S.~Cerbelli, M.~Valorani, and M.~Giona.
\newblock The structure of slow invariant manifolds and their bifurcational
  routes in chemical kinetic models.
\newblock {\em Computers and Chemical Engineering}, 31(11):1456--1474, 2007.

\bibitem{Adrover2007}
A.~Adrover, F.~Creta, M.~Giona, and M.~Valorani.
\newblock Stretching-based diagnostics and reduction of chemical kinetic models
  with diffusion.
\newblock {\em Journal of Computational Physics}, 225:1442--1471, 2007.

\bibitem{Al-Khateeb2009}
Ashraf~N. Al-Khateeb, Joseph~M. Powers, Samuel Paolucci, Andrew~J. Sommese,
  Jeffrey~A. Diller, Jonathan~D. Hauenstein, and Joshua~D. Mengers.
\newblock One-dimensional slow invariant manifolds for spatially homogenous
  reactive systems.
\newblock {\em Journal of Chemical Physics}, 131(2):024118, July 2009.

\bibitem{Bodenstein1913}
Max Bodenstein.
\newblock Eine {T}heorie der photochemischen {R}eaktionsgeschwindigkeiten.
\newblock {\em Zeitschrift f{\"u}r Physikalische Chemie -- Leipzig},
  85:329--397, 1913.

\bibitem{Brons2013}
Morten Br{{\o}}ns, Mathieu Desroches, and Maciej Krupa.
\newblock {Epsilon-free curvature methods for slow-fast dynamical systems}.
\newblock Research report, 2013.

\bibitem{Kuehn2015}
{C. Kuehn}.
\newblock {\em Mutiple Time Scale Dynamics}.
\newblock Applied Mathematical Sciences. Springer, 2015.

\bibitem{Chapman1913}
David~Leonard Chapman and Leo~Kingsley Underhill.
\newblock The interaction of chlorine and hydrogen. {T}he influence of mass.
\newblock {\em Journal of the Chemical Society, Transactions}, 103:496--508,
  1913.

\bibitem{Chiavazzo2007}
Eliodoro Chiavazzo, Alexander~N. Gorban, and Iliya~V. Karlin.
\newblock Comparison of invariant manifolds for model reduction in chemical
  kinetics.
\newblock {\em Communications in Computational Physics}, 2(5):964--992, October
  2007.

\bibitem{Chiavazzo2011}
Eliodoro Chiavazzo and Ilya Karlin.
\newblock Adaptive simplification of complex multiscale systems.
\newblock {\em Physical Review E}, 83:036706, March 2011.

\bibitem{DavisSkodje1999}
Michael~J. Davis and Rex~T. Skodje.
\newblock Geometric investigation of low-dimensional manifolds in systems
  approaching equilibrium.
\newblock {\em J. Chem. Phys.}, 111(859), 1999.

\bibitem{Fenichel1972}
Neil Fenichel.
\newblock Persistence and smoothness of invariant manifolds for flows.
\newblock {\em Indiana University Mathematics Journal}, 21(3):193--226, 1972.

\bibitem{Fenichel1974}
Neil Fenichel.
\newblock Asymptotic stability with rate conditions.
\newblock {\em Indiana University Mathematics Journal}, 23(12):1109--1137,
  1974.

\bibitem{Fenichel1977}
Neil Fenichel.
\newblock Asymptotic stability with rate conditions ii.
\newblock {\em Indiana University Mathematics Journal}, 26(1):81--93, 1977.

\bibitem{Fenichel1979}
Neil Fenichel.
\newblock Geometric singular perturbation theory for ordinary differential
  equations.
\newblock {\em Journal of Differential Equations}, 31:53--98, 1979.

\bibitem{Gear2005}
C.~W. Gear, T.~J. Kaper, I.~G. Kevrekidis, and A.~Zagaris.
\newblock Projecting to a slow manifold: Singularly perturbed systems and
  legacy codes.
\newblock {\em SIAM Journal on Applied Dynamical Systems}, 4(3):711--732, 2005.

\bibitem{Ginoux2009}
J.-M. Ginoux.
\newblock {\em Differential Geometry Applied to Dynamical Systems}, volume~66
  of {\em World Scientific series on nonlinear science: Monographs and
  treatises}.
\newblock World Scientific, 2009.

\bibitem{Ginoux2014}
Jean-Marc Ginoux.
\newblock The slow invariant manifold of the lorenz--krishnamurthy model.
\newblock {\em Qualitative Theory of Dynamical Systems}, 13(1):19--37, 2014.

\bibitem{Ginoux2006}
Jean-Marc Ginoux and Bruno Rossetto.
\newblock Differential geometry and mechanics: Applications to chaotic
  dynamical systems.
\newblock {\em International Journal of Bifurcation and Chaos}, 16(4):887--910,
  2006.

\bibitem{Ginoux2008}
Jean-Marc Ginoux and Bruno Rossetto.
\newblock Slow invariant manifolds as curvature of the flow of dynamical
  systems.
\newblock {\em International Journal of Bifurcation and Chaos},
  18(11):3409--3430, 2008.

\bibitem{Gorban2005}
Alexander~N. Gorban and Ilya~V. Karlin.
\newblock {\em Invariant Manifolds for Physical and Chemical Kinetics}, volume
  660 of {\em Lecture Notes in Physics}.
\newblock Springer-Verlag Berlin Heidelberg New York, 2005.

\bibitem{Heiter2017}
Pascal Heiter.
\newblock {\em Curvature based criteria for slow invariant manifold
  computation: from differential geometry to numerical software implementations
  for model reduction in hydrocarbon combustion}.
\newblock 2017.

\bibitem{Jones1995}
Christopher K. R.~T. Jones.
\newblock {\em Geometric singular perturbation theory}, pages 44--118.
\newblock Springer Berlin Heidelberg, Berlin, Heidelberg, 1995.

\bibitem{Kaper1999}
Tasso~J. Kaper.
\newblock An introduction to geometric methods and dynamical systems theory for
  singular perturbation problems.
\newblock In Jane Cronin, editor, {\em Analyzing multiscale phenomena using
  singular perturbation methods}, volume~56 of {\em Proceedings of Symposia in
  Applied Mathematics}, pages 85--124. American Mathematical Society,
  Providence, RI, 1999.

\bibitem{Kevrekidis2003}
Ioannis~G. Kevrekidis, C.~William Gear, James~M. Hyman, Panagiotis~G.
  Kevrekidis, Olof Runborg, and Constantinos Theodoropoulos.
\newblock Equation-free, coarse-grained multiscale computation: Enabling
  microscopic simulators to perform system-level analysis.
\newblock {\em Communications in Mathematical Sciences}, 1(4):715--762, 2003.

\bibitem{Lam1985}
S.~H. Lam.
\newblock Singular perturbation for stiff equations using numerical methods.
\newblock In Corrado Casci and Claudio Bruno, editors, {\em Recent Advances in
  the Aerospace Sciences}, pages 3--20. Plenum Press, New~York, London, 1985.

\bibitem{Lam1994}
S.~H. Lam and D.~A. Goussis.
\newblock The {CSP} method for simplifying kinetics.
\newblock {\em International Journal of Chemical Kinetics}, 26:461--486, 1994.

\bibitem{Lebiedz2004}
D.~Lebiedz.
\newblock {Computing Minimal Entropy Production Trajectories: An Approach to
  Model Reduction in Chemical Kinetics}.
\newblock {\em Journal of Chemical Physics}, 120:6890--6897, 2004.

\bibitem{Lebiedz2013}
D.~Lebiedz and J.~Siehr.
\newblock Simplified reaction models for combustion in gas turbine combustion
  chambers.
\newblock In J.~Janicka, A.~Sadiki, M.~Sch{\"a}fer, and C.~Heeger, editors,
  {\em Flow and Combustion in Advanced Gas Turbine Combustors}, chapter~5,
  pages 161--182. Springer Netherlands, Dordrecht, 2013.

\bibitem{Lebiedz2014}
D.~Lebiedz and J.~Siehr.
\newblock An optimization approach to kinetic model reduction for combustion
  chemistry.
\newblock {\em Flow, Turbulence and Combustion}, 92(4):885--902, 2014.

\bibitem{Lebiedz2016}
D.~Lebiedz and J.~Unger.
\newblock On unifying concepts for trajectory-based slow invariant attracting
  manifold computation in kinetic multiscale models.
\newblock {\em Mathematical and Computer Modelling of Dynamical Systems},
  22(2):87--112, 2016.

\bibitem{Lebiedz2006b}
Dirk Lebiedz, Volkmar Reinhardt, and Julia Kammerer.
\newblock Novel trajectory based concepts for model and complexity reduction in
  (bio)chemical kinetics.
\newblock In A.~N. Gorban, N.~Kazantzis, I.~G. Kevrekidis, and
  C.~Theodoropoulos, editors, {\em Model reduction and coarse-graining
  approaches for multi-scale phenomena}, pages 343--364. Springer, Berlin,
  2006.

\bibitem{Lebiedz2010}
Dirk Lebiedz, Volkmar Reinhardt, and Jochen Siehr.
\newblock Minimal curvature trajectories: Riemannian geometry concepts for slow
  manifold computation in chemical kinetics.
\newblock {\em Journal of Computational Physics}, 229(18):6512--6533, September
  2010.

\bibitem{Lebiedz2011}
Dirk Lebiedz, Volkmar Reinhardt, Jochen Siehr, and Jonas Unger.
\newblock Geometric criteria for model reduction in chemical kinetics via
  optimization of trajectories.
\newblock In Alexander~N. Gorban and Dirk Roose, editors, {\em Coping with
  Complexity: Model Reduction and Data Analysis}, number~75 in Lecture Notes in
  Computational Science and Engineering, pages 241--252. Springer, Heidelberg,
  first edition, 2011.

\bibitem{Lebiedz2013a}
Dirk Lebiedz and Jochen Siehr.
\newblock A continuation method for the efficient solution of parametric
  optimization problems in kinetic model reduction.
\newblock {\em SIAM Journal on Scientific Computing}, 35(3):A1548--A1603, 2013.

\bibitem{Lebiedz2011a}
Dirk Lebiedz, Jochen Siehr, and Jonas Unger.
\newblock A variational principle for computing slow invariant manifolds in
  dissipative dynamical systems.
\newblock {\em SIAM Journal on Scientific Computing}, 33(2):703--720, 2011.

\bibitem{Maas1992}
U.~Maas and S.~B. Pope.
\newblock Simplifying chemical kinetics: Intrinsic low-dimensional manifolds in
  composition space.
\newblock {\em Combustion and Flame}, 88:239--264, 1992.

\bibitem{Mease2016}
K.D. Mease, U.~Topcu, E.~Aykutlu{\u g}, and M.~Maggia.
\newblock Characterizing two-timescale nonlinear dynamics using finite-time
  lyapunov exponents and subspaces.
\newblock {\em Communications in Nonlinear Science and Numerical Simulation},
  36:148 -- 174, 2016.

\bibitem{Michaelis1913}
L.~Michaelis and M.~L. Menten.
\newblock Die {K}inetik der {I}nvertinwirkung.
\newblock {\em Biochemische Zeitschrift}, 49:333--369, 1913.

\bibitem{Ren2006a}
Z.~Ren, S.~B. Pope, A.~Vladimirsky, and J.~M. Guckenheimer.
\newblock The invariant constrained equilibrium edge preimage curve method for
  the dimension reduction of chemical kinetics.
\newblock {\em Journal of Chemical Physics}, 124:114111, 2006.

\bibitem{Ren2005}
Z.~Ren and S.B. Pope.
\newblock Species reconstruction using pre-image curves.
\newblock In {\em Proceedings of the Combustion Institute}, volume~30, pages
  1293--1300, 2005.

\bibitem{Roussel2012}
Marc~R. Roussel.
\newblock Further studies of the functional equation truncation approximation.
\newblock {\em Canadian Applied Mathematics Quarterly}, 20(2):209--227, 2012.

\bibitem{Roussel2006}
Marc~R. Roussel and Terry Tang.
\newblock The functional equation truncation method for approximating slow
  invariant manifolds: A rapid method for computing intrinsic low-dimensional
  manifolds.
\newblock {\em Journal of Chemical Physics}, 125:214103, 2006.

\bibitem{Theodoropoulos2000}
Constantinos Theodoropoulos, Yue-Hong Qian, and Ioannis~G. Kevrekidis.
\newblock Coarse stability and bifurcation analysis using time-steppers: A
  reaction-diffusion example.
\newblock {\em Proceedings of the National Academy of Sciences},
  97(18):9840--9843, 2000.

\bibitem{Transtrum2014}
Mark~K Transtrum and Peng Qiu.
\newblock Model reduction by manifold boundaries.
\newblock {\em Physical review letters}, 113(9):098701, 2014.

\bibitem{Transtrum2016}
Mark~K. Transtrum and Peng Qiu.
\newblock Bridging mechanistic and phenomenological models of complex
  biological systems.
\newblock {\em PLOS Computational Biology}, 12(5):1--34, 05 2016.

\bibitem{Valorani2009}
M.~Valorani and S.~Paolucci.
\newblock The g-scheme: A framework for multi-scale adaptive model reduction.
\newblock {\em Journal of Computational Physics}, 228(13):4665--4701, 2009.

\bibitem{Zagaris2009}
Antonios Zagaris, C.~William Gear, Tasso~Joost Kaper, and Yannis~G. Kevrekidis.
\newblock Analysis of the accuracy and convergence of equation-free projection
  to a slow manifold.
\newblock {\em ESAIM: Mathematical Modelling and Numerical Analysis},
  43(4):757--784, 2009.

\end{thebibliography}

\end{document}